\documentclass{psp}
\usepackage{amsmath,amsfonts,amssymb,amsbsy,mathrsfs,bbm,url}

\newif\ifpost   
\postfalse

\ifpost
\else
\fi

\usepackage[T1]{fontenc}
\usepackage{etoolbox}
\makeatletter
\patchcmd{\@thm}{\thm@headfont{\scshape}}{\thm@headfont{\scshape\bfseries}}{}{}
\patchcmd{\@thm}{\thm@notefont{\fontseries\mddefault\upshape}}{}{}{}
\makeatother

\newtheorem{thm}{Theorem}
\newtheorem{lem}{Lemma}[section]

\newtheorem{cor}[thm]{Corollary}

\numberwithin{equation}{section}

\newcommand{\NN}{{\mathbb N}}
\newcommand{\er}{\mathrm{e}}

\newcommand{\PR}{\mathbb P}
\newcommand{\E}{{\mathbb E}\,}
\newcommand\Pois{\operatorname{Pois}}  
\newcommand{\deq}{\overset{d}=}  

\newcommand{\bal}{\[\begin{aligned}}
\newcommand{\eal}{\end{aligned}\]}

\newcommand{\be}{\begin{equation}}
\newcommand{\ee}{\end{equation}}

\newcommand{\ssum}[1]{\sum_{\substack{#1}}}  
\newcommand{\sprod}[1]{\prod_{\substack{#1}}}  

\renewcommand{\a}{\ensuremath{\alpha}}
\renewcommand{\b}{\ensuremath{\beta}}

\newcommand{\eps}{\ensuremath{\varepsilon}}

\renewcommand{\le}{\leqslant}

\renewcommand{\ge}{\geqslant}

\newcommand{\fl}[1]{{\ensuremath{\left\lfloor {#1} \right\rfloor}}}

\renewcommand{\(}{\left(}
\renewcommand{\)}{\right)}

\newcommand{\pfrac}[2]{\left(\frac{#1}{#2}\right)}  

\newcommand{\XX}{\ensuremath{\mathbf{X}}}
\newcommand{\VV}{\ensuremath{\mathbf{V}}}

\newcommand{\one}{\ensuremath{\mathbbm{1}}} 

\begin{document}

\title{Joint Poisson distribution of prime factors in sets}
\author[Kevin Ford]{Kevin Ford \\
Department of Mathematics, 1409 West Green Street, \addressbreak University
of Illinois at Urbana-Champaign, Urbana, IL 61801, USA \addressbreak
ford@math.uiuc.edu}
\maketitle

\begin{abstract}
Given disjoint subsets $T_1,\ldots,T_m$ of ``not too large'' primes up to
$x$, we establish that for a random integer $n$ drawn from $[1,x]$,
the $m$-dimensional vector
enumerating the number of prime factors of $n$ from $T_1,\ldots,T_m$
converges to a vector of $m$ independent Poisson
random variables.  We give a specific rate of convergence
using the Kubilius model of prime factors.
We also show a universal upper bound of Poisson type when
$T_1,\ldots,T_m$ are unrestricted, and apply this to
the distribution of the number of prime factors from
a set $T$ conditional on $n$ having $k$ total prime factors.
\end{abstract}
\thanks{The author thanks G\'erald Tenenbaum and the anonymous referee for helpful comments.}
\thanks{The author was supported by NSF grant DMS-1802139.}




\section{Introduction}
A central theme in probabilistic number theory concerns the distribution
of additive arithmetic functions, in particular the functions
$\omega(n)$ and $\Omega(n)$, which count the number of
distinct prime factors of $n$ and the number of prime power factors of $n$,
respectively.
Taking a uniformly random integer $n\in [1,x]$ with $x$ large,
the functions $\omega(n)$ and $\Omega(n)$ behave like
Poisson random variables with parameter $\log\log x$.
This was established by Sathe \cite{sathe} and Selberg \cite{selberg} in 1954,
while hints of this were already present in the inequalities of
Landau \cite{Landau}, 
 Hardy and Ramanujan \cite{HR17},
Erd\H os \cite{erdos37}, and Erd\H os and Kac
\cite{erdoskac}. 
We refer the reader to Elliott's notes \cite[pp. 23--26]{elliott} for an extensive discussion of
the history of these results.

In this paper we address the distribution of the 
number of prime factors of $n$ lying in an arbitrary set $T$.
Denote by $\PR_x$ the probabiliy 
with respect to a  uniformly random integer $n$ drawn from $[1,x]$.
Each such $n$ has a unique prime factorization
\[
n= \prod_{p\le x} p^{v_p},
\]
where the exponents $v_p$ are now random variables.
For any finite set $T$ of primes, let
\[
\omega(n,T) = \# \{p|n:p\in T\}=\#\{p\in T: v_p>0\},
\qquad \Omega(n,T)=\sum_{p\in T} v_p.
\]
For a prime $p$, the event $\{p|n\}$ occurs with probability
close to $1/p$, and thus heuristically
\be\label{Poisson-approx}
\PR_x(\omega(n,T)=k) \approx \ssum{p_1,\ldots,p_k\in T \\ p_1<\cdots<p_k} \frac{1}{p_1\cdots p_k} 
\sprod{p\in T \\ p\not\in \{p_1,\ldots,p_k\}} \(1-\frac{1}{p}\) \approx \er^{-H(T)}\frac{H(T)^k}{k!}
\ee
where 
\[
H(T) = \sum_{p\in T} \frac{1}{p}.
\]
That is, we expect that
$\omega(n,T)$ will be close to Poisson with parameter $H(T)$.
A more complicated combinatorial heuristic also suggests
that $\Omega(n,T)$ is close to Poisson with parameter $H(T)$.
This was made rigorous by 
Hal{\'a}sz \cite{Halasz} in 1971, who showed\footnote{As usual, the notations
$f=O(g)$, $f\ll g$ and $g\gg f$ means that there is a constant $C$
so that $|f|\le g$ throughout the domain of $f$.  The constant $C$ is
indepenedent of any variable or parameter unless that dependence is
specified by a subscript, e.g. $f=O_A(g)$ means that $C$ depends on $A$.}
\be\label{halasz-local}
\PR_x(\Omega(n,T)=k) = \frac{H(T)^k}{k!} \er^{-H(T)}
\bigg( 1 + O_\delta\pfrac{|k-H(T)|}{H(T)} +O_\delta\bigg(\frac{1}{\sqrt{H(T)}}\bigg) \bigg),
\ee
uniformly in the range $\delta H(T) \le k\le (2-\delta) H(T)$, where
$\delta>0$ is fixed.
Small modifications to the proof yield an identical
estimate for $\PR_x(\omega(n,T)=k)$; see
\cite[p. 301]{elliott} for a sketch of the argument.
Inequality \eqref{halasz-local} implies the order of magnitude estimate
\[
\frac{H(T)^k}{k!} \er^{-H(T)} \ll
\PR_x(\Omega(n,T)=k) \ll \frac{H(T)^k}{k!} \er^{-H(T)}
\]
when $(1-\eps)H(T) \le k\le (2-\delta) H(T)$ for sufficiently small $\eps>0$.  The range of $k$
in this last bound was extended to $\delta H(T) \le k\le  (2-\delta)H(T)$
by S\'ark\H ozy \cite{Sarkozy} in 1977.

Inequality \eqref{halasz-local} implies that $\Omega(n,T)$
converges to the Poisson distribution with parameter $H(T)$ if
$T$ is a function of $x$ such that $H(T)\to \infty$ 
as $x\to\infty$.  This is a natural condition, as the following examples show.
If $T$ consists only of small primes, say those less than a bounded quantity $t$,
 then $\omega(n,T)$ takes only finitely many values
and thus the distribution cannot converge to Poisson as $x\to\infty$. 
Although $\Omega(n,T)$ is unbounded, the distribution is very far from Poisson,
e.g. $\PR_x(\Omega(n,\{2\})=k) \sim 1/2^{k+1}$ for each $k$. 
Likewise, if $c>1$ is fixed and $T$ is the
set of primes in $(x^{1/c},x]$, 
$\omega(n,T)$ and $\Omega(n,T)$ are each bounded by $c$.
Moreover, the distribution of the largest prime factors of an integer
is governed by the very different Poisson-Dirichlet distribution; see \cite{ten00}
for details.
In each of these examples, $H(T)$ is bounded.
The condition $H(T)\to \infty$
ensures that neither small primes nor large primes dominate $T$
with respect to the harmonic measure.

An asymptotic for the joint local limit laws $\PR(\omega(n;T_1)=k_1,\omega(n;T_2)=k_2)$
was proved by Delange \cite[Section 6.5.3]{Delange} in 1971, in the special case when $T_1$ and $T_2$
are infinite sets with $H(T_j \cap [1,x])=\lambda_j \log\log x + O(1)$
and $\lambda_1,\lambda_2$ constants.
Hal{\'a}sz' result \eqref{halasz-local} was extended  by Tenenbaum \cite{ten17} in 2017 to the joint distribution of $\omega(n;T_j)$
uniformly over any disjoint sets $T_1,\ldots,T_m$ of the primes $\le x$.  If $P= \PR_x(\omega(n,T_i)=k_i, 1\le i\le m)$, then
\be\label{tenenbaum}
\begin{split}
P &= 
\bigg(1+O\bigg( \sum_{j=1}^m \frac{1}{\sqrt{H(T_j)}} \bigg) \bigg)\bigg( \prod_{j=1}^m\frac{H(T_j)^{k_j}}{k_j!} \er^{-k_j} \bigg)
\frac{1}{x}\sum_{n\le x} \prod_{j=1}^m (k_j/H(T_j))^{\omega(n;T_j)}
\\
&=\prod_{j=1}^m\frac{H(T_j)^{k_j}}{k_j!} \er^{-H(T_j)}
\exp\bigg( O\bigg( \sum_{j=1}^m \frac{|k_j-H(T_j)|}{H(T_j)} +\frac{1}{\sqrt{H(T_j)}} \bigg) \bigg),
\end{split}\ee
uniformly in the range $c_1 \le k_j/H(T_j) \le c_2$
$(1\le j\le m)$,
for any fixed $c_1,c_2$ satisfying $0<c_1<c_2$;
see \cite{ten17}, equation (2.23) and the following paragraph.
The methods in \cite{ten17} establish the same
bound for $\PR_x(\Omega(n,T_i)=m_i, 1\le i\le k)$,
but with the restriction
$c_1 \le \frac{k_j}{H(T_j)} \le 2-c_1$, $1\le j\le m$, again with fixed $c_1>0$.  An asymptotic for the sum on $n$ in \eqref{tenenbaum} is not known
in general.
A slight extension of Tenenbaum's asymptotic \eqref{tenenbaum} was given by
Mangerel \cite[Theorem 1.5.3]{Mangerel}, who showed a corresponding asymptotic
in the case where some of the quantities $k_j$ are smaller (specifically, $H(T_j)^{2/3+\eps}<k_j\le H(T_j)$).

In the literature on the subject, $\omega(n,T)$ and 
$\Omega(n,T)$ have always been compared to a 
Poisson variable with parameter $H(T)$.  
As we shall see, the functions $\Omega(n,T)$
are better approximated by a Poisson variable with
parameter
\[
H'(T) = \sum_{p\in T} \frac{1}{p-1},
\]
at least when $T$ does not contain any large primes.
In order to state our results, we introduce a further harmonic sum
\[
H''(T) = \sum_{p\in T} \frac{1}{p^2}.
\]
We note for future reference that
\[
H(T) \le H'(T) \le H(T) + 2 H''(T).
\]
We also use the notion of the total
variation distance $d_{TV}(X,Y)$ between two random variables
 living on the same discrete space $\Omega$:
 \[
 d_{TV}(X,Y) := \sup_{A\subset \Omega} \big| \PR(X\in A)-\PR(Y\in A) \big|.
 \]
We denote by $\Pois(\lambda)$ a Poisson random variable
with parameter $\lambda$, and write $Z\deq \Pois(\lambda)$
for the statement that $Z$ is a Poisson random variable
with parameter $\lambda$.

\begin{thm}\label{thm-main}
Let $2\le y\le x$ and suppose that
 $T_1,\ldots,T_m$ are disjoint nonempty sets of primes in $[2,y]$.
 For each $1\le i\le m$, suppose that either
 $f_i=\omega(n,T_i)$ and $Z_i\deq \Pois(H(T_i))$ or that
 $f_i=\Omega(n,T_i)$ and $Z_i\deq \Pois(H'(T_i))$.  
Assume that $Z_1,\ldots,Z_m$
 are independent.
Then
\[
d_{TV} \Big( (f_1,\ldots,f_m), (Z_1,\ldots,Z_m) \Big) \ll  \sum_{j=1}^m \frac{H''(T_j)}{1+H(T_j)} + u^{-u}, \quad u=\frac{\log x}{\log y}.
\]
\end{thm}

The implied constant is absolute, independent of
$m$, $y$, $x$ and $T_1,\ldots,T_m$.  In particular, if $m$ is fixed 
then this shows
that the joint distribution of  $(f_1,\ldots,f_m)$ converges
to a joint Poisson distribution whenever we have $y=x^{o(1)}$ and for each $i$,
either $H(T_i)\to \infty$ or $\min T_i \to \infty$.

By contrast, Tenenbaum's bound \eqref{tenenbaum} implies
\be\label{ten-TV}
d_{TV} \Big( (\omega(n,T_1),\ldots,\omega(n,T_m)),(Z_1,\ldots,Z_m) \Big) \ll_m \sum_{j=1}^m \frac{1}{\sqrt{H(T_j)}}.
\ee
Compared to Theorem \ref{thm-main},
we see that \eqref{ten-TV}
gives good results even if the sets $T_i$ contain many large primes,
while Theorem \ref{thm-main} requires that
$y\le x^{o(1)}$ in order to be nontrivial.
However, if $y\le x^{1/\log\log\log x}$, say,
the conclusion of Theorem \ref{thm-main} is stronger,
especially
when $H''(T)$ is small.
An extreme case is given by singleton set $T=\{p\}$ and $f_1=\Omega(n,T)$,
where Theorem \ref{thm-main} recovers the correct order
of $d_{TV}(f_1,Z_1)$, namely $1/p^2$,
since $\PR_x(p\|n) \approx \frac{1}{p}-\frac{1}{p^2}$, $\PR_x(p^2\|n)\approx \frac{1}{p^2}-\frac{1}{p^3}$, and
$\PR(Z_1=2) \approx 1/(2p^2)$ for large $p$.

\medskip

\noindent
\textbf{Example.}
Let $S$ be the set of all primes,
$t_k=\exp\exp k$ and
$\omega_k(n) := \omega(n,S \cap (t_k,t_{k+1}])$.
Here, by the Prime Number Theorem with strong error term,
\[
H(S \cap (t_k,t_{k+1}]) = 1 + O(\exp \{ - \er^{k/2} \}).
\]
Thus, $\omega_k$ has distribution close to that of a
Poisson variable with parameter 1.  More precisely, if 
$X,Y$ are Poisson with parameters $\lambda,\lambda'$, respectively,
then (e.g. \cite[Theorem 1.C, Remark 1.1.2]{BHJ})
\[
d_{TV}(X,Y) \le |\lambda-\lambda'|.
\]

Using a standard inequality for $d_{TV}$ (\eqref{dTV-vec} below), we deduce the following.

\begin{cor}\label{thm:Poisson-1}
If $\xi \le k < \ell \le \log\log x - \xi$, then
\be\label{Poisson-1}
d_{TV}\big( (\omega_k,\ldots,\omega_\ell),(Z_k',\ldots,Z_\ell') \big)
\ll \exp \{-\er^{\xi/2}\},
\ee
where $Z_k',\ldots,Z_\ell'$ are independent Poisson variables with parameter
1. 
\end{cor}

Thus, statistics of the random \emph{function}
$f(t)=\omega(n,S\cap [t_k,t])$, $t_k\le t\le t_\ell$, are captured very accurately
by statistics of the partial sums $Z_k'+\cdots+Z_m'$ for $k\le m\le \ell$.
The latter has been well-studied
and one can easily deduce, for example, the Law of the Iterated
Logarithm for $f(t)$ from that for the partial sums 
$Z_k'+\cdots + Z_\ell'$.
Similarly, if $T$ is a set of primes with density $\a>0$
in the sense that
\[
\sum_{p\le x,p\in T} \frac{1}{p} = \a \log\log x + c + o(1)
\quad (x\to \infty)
\]
then a statement similar to \eqref{Poisson-1} holds
 with $t_k$ replaced
by $t_k'=\exp\exp (k/\a)$, with a weaker estimate for the
total variation distance
(depending on the decay of the $o(1)$ term).

%
%

Next, we establish the upper-bound implied in
 \eqref{tenenbaum}, but valid
uniformly for all $k_1,\ldots,k_m$. 

\begin{thm}\label{prime_factors_sets}
Let $T_1,\ldots,T_r$ be arbitrary disjoint, nonempty subsets of the primes $\le x$.
For any $k_1,\ldots,k_r\ge 0$, letting
$P = \PR_x \big( \omega(n;T_j) = k_j \; (1\le j\le r) \big)$, we have
\begin{align*}
P &\ll  \prod_{j=1}^r \Bigg(\frac{H'(T_j)^{k_j}}{k_j!} \er^{-H(T_j)} \Bigg) \( \eta + 
\frac{k_1}{H'(T_1)}+\cdots+\frac{k_r}{H'(T_r)}\)+ \xi\\
&\le  \prod_{j=1}^r \Bigg(\frac{(H(T_j)+2)^{k_j}}{k_j!} \er^{-H(T_j)} \Bigg),
\end{align*}
where $\eta=0$ if $T_1\cup \cdots \cup T_r$ contains every prime $\le x$
and $\eta=1$ otherwise, and $\xi=1$ if $\eta=k_1=\cdots=k_r=0$ and $\xi=0$
otherwise.
\end{thm}

\noindent
\textbf{Remarks.}
Tudesq \cite{tudesq} claimed a bound similar to Theorem \ref{prime_factors_sets},
but only supplied details for $r=1$.
Our method is similar, and we give a short, complete
proof in Section \ref{sec:primes_sets}.

If we condition on $\omega(n)=k$, 
 the $r=2$ case
of Theorem \ref{prime_factors_sets} supplies tail
bounds for $\omega(n,T)$.
If $X,Y$ are independent Poisson random variables with parameters
$\lambda_1,\lambda_2$, respectively, then for $0\le \ell\le k$, we have
\[
\PR ( X = \ell | X+Y = k) = \binom{k}{l} \pfrac{\lambda_1}{\lambda_1+\lambda_2}^\ell
\pfrac{\lambda_2}{\lambda_1+\lambda_2}^{k-\ell}.
\]
Thus, conditional on $\omega(n)=k$ 
we expect that $\omega(n,T)$ will have roughly a binomial
distribution with parameter $\a=H(T)/H(S)$,
where $S$ is the set of all primes in $[2,x]$.

\begin{thm}\label{conditional}
Fix $A>1$ and suppose that $1 \le k\le A\log\log x$.
Let $T$ be a nonempty subset of the primes in $[2,x]$
and define let $\a=H(T)/H(S)$.
For any $0\le \psi \le \sqrt{\a k}$ we have
\[
\PR \Big( |\omega(n,T)-\a k| \ge \psi\sqrt{\a(1-\a) k}\; \Big| \; \omega(n)=k \Big) \ll_A \er^{-\frac13 \psi^2},
\]
the implied constant depending only on $A$.
\end{thm}

Similarly, if $T_1,\ldots,T_m$ are disjoint subsets of primes $\le x$
and we condition on $\omega(n)=k$, then the vector
$(\omega(n,T_1),\ldots,\omega(n,T_m))$ will have approximately
a multinomial distribution.


\section{The Kubilius model of small prime factors of integers}

Our restriction to  primes below $x^{o(1)}$
comes from an application of a probabilistic model of
prime factors, called the Kubilius model, 
and introduced by
Kubilius \cite{kub56, kubilius} in 1956.
We compute
\[
\PR_x (v_p=k) = \frac{1}{\fl{x}} \( \fl{\frac{x}{p^k}} - \fl{\frac{x}{p^{k+1}}} \)  = \frac{1}{p^k} - \frac{1}{p^{k+1}} + O\pfrac{1}{x},
\]
the error term being relatively small when $p^k$ is small.
Moreover, the variables $v_p$ are quasi-independent; that is, the 
correlations are small, again provided that the primes are small.
By contrast,
the variables $v_p$ corresponding to large $p$ are very much dependent, for example the event $(v_p>0, v_q>0)$ is impossible if $pq>x$.

The model of Kubilius is a sequence of \emph{idealized} random variables 
which removes the error term above, and is much easier to compute with.
For each prime $p$, define the random variable $X_p$ that has domain $\NN_0=\{0,1,2,3,4,\ldots\}$
and such that
\[
\PR (X_p=k) = \frac{1}{p^k} - \frac{1}{p^{k+1}} = \frac{1}{p^k} \(1-\frac{1}{p} \) \qquad (k=0,1,2,\ldots).
\]

The principal result, first proved by Kubilius and
later sharpened by others, is that  the random vector
\[
\XX_y = (X_p:p\le y)
\]
has distribution close to that of the random vector
\[
\VV_{x,y} = (v_p: p\le y),
\]
provided that $y=x^{o(1)}$.

In \cite{ten99}, Tenenbaum gives a rather complicated 
asymptotic for $d_{TV}(\XX_y,\VV_{x,y})$
in the range $\exp\{(\log x)^{2/5+\eps}\}\le y\le x$,
as well as a simpler universal upper bound which we state here.

\begin{lem}[{Tenenbaum \cite[Th\'eor\`eme 1.1 and (1.7)]{ten99}}]\label{Kubilus}
Let $2\le y\le x$.  Then, for every $\eps>0$,
\[
d_{TV}(\XX_y,\VV_{x,y}) \ll_\eps u^{-u} + x^{-1+\eps}, \quad u=\frac{\log x}{\log y}.
\]
\end{lem}

%
%
\section{Poisson approximation of prime factors}

  For a finite set $T$ of primes, denote
\[
U_T = \# \{p\in T: X_p\ge 1 \}, \qquad
W_T = \sum_{p\in T} X_p,
\]
which are probabilistic models
for $\omega(n,T)$ and 
 $\Omega(n,T)$, respectively.
 For any $T$
which is a subset of the primes $\le y=x^{1/u}$, 
Lemma \ref{Kubilus} implies that for any $\eps>0$,
\be\label{Kubilius-omega}
\begin{split}
d_{TV}(U_T,\omega(n,T)) &\ll_\eps  u^{-u} + x^{-1+\eps},\\
\quad d_{TV}(W_T,\Omega(n,T)) &\ll_\eps  u^{-u} + x^{-1+\eps}.
\end{split}
\ee 

 We next prove a local limit theorem for $U_T$ and $W_T$,
 and then use this to establish Theorem \ref{thm-main}.

 \begin{thm}\label{kubilius_poisson}
 Let $T$ be a finite subset of the primes,
 and let $Y=U_T$ or $Y=W_T$. 
Let $H=H(T)$ if $Y=U_T$ and $H=H'(T)$
if $Y=W_T$.  Also let $Z\deq \Pois(H)$.
 Then
 \[
 \PR \( Y = k \) - \PR (Z=k ) \ll \begin{cases}
 H''(T) \frac{H^k}{k!}\er^{-H} \(\frac{1}{k+1}+\pfrac{k-H}{H}^{2} \) & \text{ if } 0\le k\le 1.9H \\  H''(T)\pfrac{\er^{0.9 H}}{(1.9)^k} & \text{ if } k>1.9H.
 \end{cases}
\] 
\end{thm}

\begin{proof}
Write $H''=H''(T)$.
When $k=0$, $\PR(Z=0)=\er^{-H}$ and
\[
\PR(Y=0)=\PR(\forall p\in T: X_p=0)=\prod_{p\in T}\(1-\frac{1}{p}\) = \er^{-H}(1+O(H'')),
\]
and the desired inequality follows.

For $k\ge 1$,
we work with moment generating functions
as in the proof of
Hal{\'a}sz' theorem \eqref{halasz-local}; see also \cite[Ch. 21]{elliott}.
For any complex $z$, 
\[
\E z^{Z} = \er^{(z-1)H}.
\]
Uniformly for complex $z$ with $|z|\le 2$ we have
\be\label{EUT}
\E z^{U_T} = \prod_{p\in T} \(1+ \frac{z-1}p \) = \er^{(z-1)H(T)}
\Big(1 + O\big( |z-1|^2 H''(T) \big) \Big)
\ee
and uniformly for $|z| \le 1.9$ we have
\be\label{EWT}
\E z^{W_T} = \prod_{p\in T} \(1+ \frac{z-1}{p-z} \)=
 \er^{(z-1)H'(T)} \big(1 + O( |z-1|^2 H''(T) ) \big).
\ee
Write $e(\theta) = \er^{2\pi i \theta}$.
Then, for any $0 < r \le 1.9$, \eqref{EUT} and \eqref{EWT} imply
\bal
\PR(Y=k)-\PR(Z=k) &=\frac{1}{2\pi i} \oint\limits_{|z|=r} \frac{\E z^{Y} - \E z^{Z}}{z^{k+1}}\, dw \\
&= \frac{1}{r^{k}} \int_0^1 e(-k\theta)
\Big[ \E (re(\theta))^{Y} - \E (re(\theta))^{Z} \Big] \,d\theta \\
&= \frac{1}{r^k} \int_0^1  e(-k\theta) \er^{(re(\theta)-1)H} \cdot O \( |re(\theta)-1|^2 H'' \)\,  d\theta \\
&\ll \frac{H''}{r^k} \int_0^{1/2} |r e(\theta)-1|^2 \er^{(r\cos (2\pi \theta)-1) H}\, d\theta.
\eal
Now, for $0\le \theta \le \frac12$,
 $$r\cos (2\pi\theta)-1 = r-1-2r\sin^2(\pi \theta) \le r-1-8r\theta^2$$
and
\[
|re(\theta)-1|^2 = (r-1-2r\sin^2(\pi \theta))^2 + \sin^2(2\pi \theta) \ll (r-1)^2 + \theta^2,
\]
so we obtain
\be\label{UZ_first}
\begin{split}
\PR(Y=k)-\PR(Z=k) &\ll H'' \frac{\er^{(r-1)H}}{r^k} \int_0^{1/2} (|r-1|^2+\theta^2) \er^{-8r\theta^2 H}\, d\theta \\
 &\ll H'' \frac{\er^{(r-1)H}}{r^k}
\bigg( \frac{|r-1|^2}{\sqrt{1+rH}} + \frac{1}{(1+rH)^{3/2}}  \bigg).
\end{split}\ee

When $1\le k\le 1.9H$, we take $r = k/H$ in \eqref{UZ_first} and obtain, using Stirling's formula,
\bal
\PR(Y=k)-\PR(Z=k) &\ll H'' \frac{H^k \er^{k-H}}{k^k} \( \frac{|k/H-1|^2}{k^{1/2}}+ \frac{1}{k^{3/2}} \)\\
&\ll H'' \frac{\er^{-H} H^k}{k!} \( \bigg|\frac{k-H}{H}\bigg|^2 + \frac{1}{k} \).
\eal
When $k>1.9H$, take $r=1.9$ in \eqref{UZ_first}
and conclude that
\[
\PR(Y=k)-\PR(Z=k) \ll \frac{H'' \er^{0.9 H}}{(1.9)^k\sqrt{1+H}}.
\]
This completes the proof.
\end{proof}

\begin{cor}\label{dTV-cor}
Let $T$ be a finite  subset of the primes.  Then
\[
d_{TV}(U_T,\Pois(H(T))) \ll \frac{H''(T)}{1+H(T)}
\]
and
\[
d_{TV}(W_T,\Pois(H'(T))) \ll \frac{H''(T)}{1+H(T)},
\]
 \end{cor}
 
\begin{proof}
Let $Y \in \{U_T,W_T\}$.  If $Y=U_T$, let $H=H(T)$
and if $Y=W_T$, let $H=H'(T)$.  Let $Z\deq \Pois(H)$.
Again, write $H''=H''(T)$.
We begin with the identity
\[
d_{TV}(Y,Z) = \frac12 \sum_{k=0}^\infty \big|
\PR(Y_T=k) - \PR(Z(T)=k) \big|.
\]
Consider two cases.  First, if $H\le 2$, we have by Theorem \ref{kubilius_poisson},
\[
\sum_{k\ge 0} |\PR(Y=k)-\PR(Z=k)| \ll H'' + \sum_{k >  1.9H} H'' (1.9)^{-k} \ll H''.
\]
If $H>2$, Theorem \ref{kubilius_poisson} likewise implies that
\begin{align*}
\sum_{k > 1.9 H} |\PR(Y=k)-\PR(Z=k)| \ll H'' \sum_{k >  1.9H} \frac{\er^{0.9 H}}{(1.9)^k} 
\ll H'' \er^{-0.3 H}
\end{align*}
and also
 \bal
 \sum_{k \le 1.9H} |\PR(Y=k)-\PR(Z=k)| &\ll H'' \er^{-H} \sum_{k\le 1.9H} \frac{H^k}{k!}\Bigg[ \frac{1}{k+1} + \bigg|\frac{k-H_1}{H}\bigg|^2 \Bigg]\\
 &\ll \frac{H''}{H} \ll \frac{H''}{H(T)},
 \eal
 using that $\er^{-H} H^k/k!$ decays rapidly for
 $|k-H| > \sqrt{H}$.
\end{proof}

We now combine Theorem \ref{kubilius_poisson} with
the standard inequality
\be\label{dTV-vec}
d_{TV} ( (X_1,\ldots,X_m),(Y_1,\ldots,Y_m) ) \le \sum_{j=1}^m d_{TV} (X_j,Y_j),
\ee
valid if $X_1,\ldots,X_m$ are independent,
 and $Y_1,\ldots,Y_m$ are independent, with all variables
 living on the same set $\Omega$.

\begin{cor}\label{prime_poisson_intervals}
Let $T_1,\ldots,T_m$ be disjoint sets of primes.
For each $i$, either let $Y_i=U_{T_i}$ and $H_i=H(T_i)$
or let $Y_i=W_{T_i}$ and $H_i=H'(T_i)$.
For each $i$, let $Z_i\deq \Pois(H_i)$, and suppose
that $Z_1,\ldots,Z_m$ are independent.
  Then
\[
d_{TV} \big( (Y_1,\ldots,Y_m), (Z_1,\ldots,Z_m) \big) \ll \sum_{j=1}^m \frac{H''(T_j)}{1+H(T_j)}. 
\]
\end{cor}

Combining Corollary \ref{prime_poisson_intervals}
with \eqref{Kubilius-omega} and the triangle inequality, we see that
\[
d_{TV} \Big( (f_1,\ldots,f_m), (Z_1,\ldots,Z_m) \Big) \ll  \sum_{j=1}^m \frac{H''(T_j)}{1+H(T_j)} + u^{-u} + x^{-0.99}.
\]
We may remove the term $x^{-0.99}$, because if
$y\le x^{1/3}$ then $H''(T_i) \gg x^{-2/3}$
and $H(T_i) \ll \log\log x$, while
 if $y > x^{1/3}$ then $u^{-u} \gg 1$.
This completes the proof of Theorem \ref{thm-main}.

%
%
\section{A uniform upper bound}\label{sec:primes_sets}
%
%

In this section we prove Theorem \ref{prime_factors_sets}
 and Theorem \ref{conditional}.

\begin{proof}[Proof of Theorem \ref{prime_factors_sets}]
Let 
\[
N = \# \{n\le x : \omega(n;T_j) = k_j\; (1\le j\le r) \} .
\]
If $\eta=0$ (that is, $T_1 \cup \cdots \cup T_r$ contains all the primes $\le x$)
and $k_1=\cdots=k_r=0$, then $N=1$; this explains the need for the additive term $\xi$
in Theorem \ref{prime_factors_sets}.  

Now assume that either $\eta=1$ or that $k_i\ge 1$ for some $i$.
Let
\[
L_t(x) = \ssum{h\le x \\ \omega(h;T_j)=k_j-\one_{j=t} \; (1\le j\le r)} \frac{1}{h} \qquad (0\le t\le r),
\]
where $\one_A$ is the indicator function of the condition $A$.
We use the ``Wirsing trick'', starting with
$\log x \ll \log n = \sum_{p^a\| n}\log p^a$ for $x^{1/3}\le n\le x$ and thus
\[
(\log x) N  \ll \ssum{n\le x^{1/3} \\ \omega(n;T_j)=k_j \; (1\le j\le r)} \log x + \ssum{n\le x \\ \omega(n;T_j) = k_j\; (1\le j\le r)} \sum_{p^a \| n} \log p^a.
\]
In the first sum, $\log x \le \frac{x^{1/3} \log x}{n} \ll \frac{x^{1/2}}{n}$, hence the sum is 
at most $\le x^{1/2} L_0(x)$.  In the double sum, let $n=p^ah$ and observe that $\omega(h,T_j)=k_j-1$ if $p\in T_j$ and 
 $\omega(h,T_j)=k_j$ otherwise.   
In particular, if $p \not \in T_1\cup \cdots \cup T_r$ then $\omega(h,T_j)=k_j$
for all $j$, and this is only possible if $\eta=1$.
 Hence
 \[
 (\log x) N 
\ll x^{1/2} L_0(x) + \sum_{t=1-\eta}^r \ssum{h\le x \\ \omega(h;T_j)=k_j-\one_{j=t} \; (1\le j\le r)} \sum_{p^a\le x/h} \log p^a.
\]
Using Chebyshev's Estimate for primes, the innermost sum over $p^a$
is $O(x/h)$ and thus the double sum over $h,p^a$ is $O(L_t(x))$.
Also, if $k_j=0$ then there is the sum corresponding to $t=j$ is empty.
 This gives
\be\label{primes-sets-1}
 \PR_x \Big(\omega(n;T_j) = k_j\; (1\le j\le r) \Big) \ll \frac{1}{\log x} \bigg(  (\eta + x^{-1/2}) L_0(x) + \sum_{1\le t\le r: k_t>0} L_t(x) \bigg). 
\ee
Now we fix $t$ and bound the sum $L_t(x)$; if $t\ge 1$ we may assume that $k_t\ge 1$.  Write the denominator $h=h_1\cdots h_r h'$, 
where, for $1\le j\le r$, $h_j$ is composed only of primes from $T_j$,
\[
\omega(h_j;T_j)=m_j := k_j-\one_{t=j},
\]
 and $h'$ is composed of primes below $x$ which lie in none 
of the sets $T_1, \cdots, T_r$.  For $1\le j\le r$ 
we have
\[
\sum_{h_j} \frac{1}{h_j} \le \frac{1}{m_j!} \bigg(
\sum_{p\in T_j} \frac{1}{p}+\frac{1}{p^2}+\cdots \bigg)^{m_j}
= \frac{H'(T_j)^{m_j}}{m_j!},
\]
and, using Mertens' estimate,
\[
\sum_{h'} \frac{1}{h'} \le \sprod{p\le x \\ p\not\in T_1\cup\cdots \cup T_r} \(1 - \frac{1}{p} \)^{-1} 
\ll (\log x) \prod_{p\in  T_1\cup\cdots \cup T_r} \(1 - \frac{1}{p} \).
\]
Thus,
\[
L_t(x) \ll (\log x) \prod_{j=1}^r \frac{H'(T_j)^{m_j}}{m_j!}  \prod_{p\in  T_1\cup\cdots \cup T_r} \(1 - \frac{1}{p} \).
\] 
Using the elementary inequality $1+y\le \er^y$, we see that
the final product over $p$ is  at most
$\er^{-H(T_1)-\cdots-H(T_r)}$, and we find that
\be\label{primes-sets-L}
L_t(x) \ll (\log x) \prod_{j=1}^r \Bigg( \frac{H'(T_j)^{m_j}}{m_j!} \er^{-H(T_j)} \Bigg)
\ee
Combining estimates \eqref{primes-sets-1} and \eqref{primes-sets-L},
we conclude that
\[
\PR_x \Big( \omega(n;T_j) = k_j \; (1\le j\le r) \Big)  \ll  
\Bigg(\eta+ x^{-1/2} + \sum_{j=1}^r \frac{k_j}{H'(T_j)} \Bigg)
\prod_{j=1}^r \Bigg( \frac{H'(T_j)^{k_j}}{k_j!} \er^{-H(T_j)} \Bigg).
\]
Either $\eta=1$ or $k_j/H'(T_j) \gg 1/\log\log x$ for some $j$, and hence the additive term $x^{-1/2}$
may be omitted.
This proves the first claim.

Next,
\[
\prod_{j=1}^r \frac{H'(T_j)^{k_j}}{k_j!} 
\Bigg(1 + \sum_{j=1}^r \frac{k_j}{H'(T_j)} \Bigg) 
\le \prod_{j=1}^r \frac{(H'(T_j)+1)^{k_j}}{k_j!} 
\]
and we have $H'(T) \le H(T)+\sum_p \frac{1}{p(p-1)}
\le H(T)+1$.  This proves the final inequality.
\end{proof}

To prove Theorem \ref{conditional}
we need standard tail bounds for
the binomial distribution.  For proofs, see  \cite[Lemma 4.7.2]{Ash} or
 \cite[Th. 6.1]{DT}.

\begin{lem}[Binomial tails]\label{binomial_tails}
Let $X$ have binomial distribution according to $k$ trials and parameter $\a\in [0,1]$;
that is, $\PR(X=m)=\binom{k}{m} \a^m (1-\a)^{k-m}$.
If $\beta\le \a$ then we have
\[
\PR(X\le \beta k) \le \exp \left\{ - k \( \b \log \frac{\b}{\a} + (1-\b) \log \frac{1-\b}{1-\a} \) \right\}
\le \exp \bigg\{-\frac{(\a-\beta)^2 k}{3\a(1-\a)} \bigg\}.
\]
Replacing $\a$ with $1-\a$ we also have for $\beta\ge \a$,
\[
\PR(X\ge \beta k) \le  \exp \bigg\{-\frac{(\a-\beta)^2 k}{3\a(1-\a)} \bigg\}.
\]
\end{lem}

\begin{proof}[Proof of Theorem \ref{conditional}]
We may assume that $\alpha k \ge C$, where 
$C$ is a sufficiently large constant, depending on $A$.
Without loss of generality, we may assume that $H(T)\le \frac12 H(S)$ (that is , $\alpha\le \frac12$), else replace $T$ by $S \setminus T$.
Apply
Theorem \ref{prime_factors_sets} with two sets: $T_1=T$
and $T_2=S\setminus T$, so that $\eta=\xi=0$.
We need the lower bound
\[
\PR_x (\omega(n)=k) \gg_A \frac{(\log\log x)^{k-1}}{(k-1)!\log x}
= \frac{k}{\log\log x} \,\cdot \,  \frac{(\log\log x)^{k}}{k!\log x}
\]
see, e.g. Theorem 6.4 in Chapter II.6 of \cite{Tenbook}.
Also,
\[
\Bigg( \frac{k-h}{H'(S\setminus T)} + \frac{h}{H'(T)} \Bigg) \frac{\log\log x}{k}
\ll 1 + \frac{h}{\alpha k}.
\]
Since $H'(S\setminus T) \le H(S\setminus T)+1$, we have
\[
H'(S\setminus T)^{k-h} \ll H(S\setminus T)^{k-h}.
\]
In addition, 
\[
H'(T)^h \le (H(T)+1)^h \le  H(T)^h \er^{h/H(T)} \le H(T)^h \er^{O_A(h/(\alpha k))}.
\]
 Then, for $0\le h\le k$, Theorem \ref{prime_factors_sets} implies
 \[
 \PR\Big( \omega(n,T)=h \big| \omega(n)=k \Big)  \ll_A
 \a^{h} (1-\a)^{k-h} \binom{k}{h} \er^{O_A(h/(\alpha k))}.
 \]
Ignoring the factor $(1-\a)^{k-h}$, we see that
the terms with $h\ge 100\alpha k$ contribute at most
\begin{align*}
\sum_{h\ge 100\a k} \frac{(\a k \er^{O_A(1/(\a k))})^h}{h!}
\le \sum_{h\ge 100\a k} \frac{(2\a k)^h}{h!}  \le \er^{-100 \a k}
\le \er^{-100 \psi^2}
\end{align*}
for large enough $C$.
When $h<100\a k$ we have
\[
\PR\Big( \omega(n,T)=h \big| \omega(n)=k \Big)  \ll_A
 \a^{h} (1-\a)^{k-h} \binom{k}{h},
\]
and the theorem now follows from Lemma \ref{binomial_tails},
taking
 $\beta  = \a \pm \psi\sqrt{\a(1-\a)/k}$.
\end{proof}

%
%

\end{document}